\documentclass[11pt]{article}

\usepackage{latexsym}
\usepackage{amssymb}
\usepackage{amsthm}
\usepackage{amsmath}

\newtheorem{theorem}{Theorem}[section]
\newtheorem{lemma}{Lemma}[section]

\newtheorem{corollary}{Corollary}[section]

\numberwithin{equation}{section}

\newcommand{\FF}{\mathbb{F}}

\newcommand{\be}{\mathbf{e}}

\newcommand{\ulk}{\underline{k}}
\newcommand{\ule}{\underline{e}}
\newcommand{\tH}{\tilde{H}}

\makeatletter
\renewcommand{\@makefnmark}{\mbox{\textsuperscript{}}}
\makeatother

\allowdisplaybreaks[1]

\begin{document}

\bibliographystyle{amsplain}

\title{Multivariate Rogers-Szeg\"o polynomials and flags in finite vector spaces}
\author {C. Ryan Vinroot} 
\date{}

\maketitle

\begin{abstract} We give a recursion for the multivariate Rogers-Szeg\"o polynomials, along with another recursive functional equation, and apply them to compute special values.  We also consider the sum of all $q$-multinomial coefficients of some fixed degree and length, and give a recursion for this sum which follows from the recursion of the multivariate Rogers-Szeg\"o polynomials, and generalizes the recursion for the Galois numbers.  The sum of all $q$-multinomial coefficients of degree $n$ and length $m$ is the number of flags of length $m-1$ of subspaces of an $n$-dimensional vector space over a field with $q$ elements.  We give a combinatorial proof of the recursion for this sum of $q$-multinomial coefficients in terms of finite vector spaces.
\end{abstract}

\section{Introduction}

For a parameter $q \neq 1$, and a positive integer $n$, let $(q)_n = (1 - q)(1 - q^2) \cdots (1-q^{n})$, and $(q)_0 = 1$.  For non-negative integers $n$ and $k$, with $n \geq k$, the {\em $q$-binomial coefficient} or {\em Gaussian polynomial}, denoted $\binom{n}{k}_q$, is defined as $\binom{n}{k}_q = \frac{(q)_n}{(q)_k (q)_{n-k}}$.

The {\em Rogers-Szeg\"o polynomial} in a single variable, denoted $H_n(t)$, is defined as
$$ H_n(t) = \sum_{k = 0}^n \binom{n}{k}_q t^k.$$
The Rogers-Szeg\"o polynomials first appeared in papers of Rogers \cite{Rog1, Rog2} which led up to the famous Rogers-Ramanujan identities, and later were independently studied by Szeg\"o \cite{Sz26}.  They are important in combinatorial number theory (\cite[Ex. 3.3--3.9]{An76} and \cite[Sec. 20]{Fi88}), symmetric function theory \cite{Wa06}, and are key examples of orthogonal polynomials \cite{AsWi85}.  They also have applications in mathematical physics \cite{Ka06, Ma89}.

The Rogers-Szeg\"o polynomials satisfy the recursion (see \cite[p. 49]{An76}) 
\begin{equation} \label{RogSze}
H_{n+1}(t) = (1+t)H_n(t) + t(q^n - 1)H_{n-1}(t).
\end{equation}
Letting $t=1$, we have $H_n(1) = \sum_{k = 0}^n \binom{n}{k}_q$, which, when $q$ is the power of a prime, is the total number of subspaces of an $n$-dimensional vector space over a field with $q$ elements.  The numbers $G_n = H_n(1)$ are the \emph{Galois numbers}, and from (\ref{RogSze}), satisfy the recursion $G_{n+1} = 2G_n + (q^n - 1)G_{n-1}$.  The Galois numbers were studied from the point of view of finite vector spaces by Goldman and Rota \cite{GoRo69}, and have been studied extensively, for example, in \cite{NiSoWi84, HiHo10}.

For non-negative integers $k_1, k_2, \ldots, k_m$ such that $k_1 + \cdots + k_m = n$, we define the $q$-multinomial coefficient of length $m$ as 
$$ \binom{n}{k_1, k_2, \ldots, k_m}_q = \frac{(q)_n}{(q)_{k_1} (q)_{k_2} \cdots (q)_{k_m}},$$
so that $\binom{n}{k}_q = \binom{n}{k, n-k}_q$.  Define the homogeneous Rogers-Szeg\"o polynomial in $m$ variables for $m \geq 2$, denoted $\tH_{n}(t_1, t_2, \ldots, t_m)$, by
$$ \tH_n(t_1, t_2, \ldots, t_m) = \sum_{k_1 + \cdots + k_m = n} \binom{n}{k_1, \ldots, k_m}_q t_1^{k_1} \cdots t_m^{k_m},$$
and define the Rogers-Szeg\"o polynomial in $m-1$ variables, denoted $H_n(t_1, \ldots, t_{m-1})$, by 
$$H_n(t_1, \ldots t_{m-1}) = \tH(t_1, \ldots, t_{m-1}, 1).$$  
The homogeneous multivariate Rogers-Szeg\"o polynomials were first defined by Rogers \cite{Rog1} in terms of their generating function, and several of their properties are given by Fine \cite[Section 21]{Fi88}.  The definition of the multivariate Rogers-Szeg\"o polynomial $H_n$ is given by Andrews in \cite[Chap. 3, Ex. 17]{An76}, along with a generating function, although there is little other study of these polynomials elsewhere in the literature (however, there is a non-symmetric version of a bivariate Rogers-Szeg\"o polynomial \cite{ChSaSu07}).  In Section \ref{recursion}, we give a recursion for the multivariate Rogers-Szeg\"o polynomials which generalizes (\ref{RogSze}).  The result, given in Theorem \ref{main} below, follows quickly from the generating function for the multivariate Rogers-Szeg\"o polynomials, although seems not to have been noted before.  We also give a few other properties of the multivariate Rogers-Szeg\"o polynomials in Section \ref{recursion} which complement those given in \cite[Section 21]{Fi88}.

Finally, in Section \ref{flags}, we concentrate on the value $H_n(1, 1, \ldots, 1)$, of the Rogers-Szeg\"o polynomial $H_n(t_1, \ldots, t_{m-1})$ when we let $t_1 = \cdots = t_{m-1} = 1$.  This is the sum of all $q$-multinomial coefficients of length $m$, which we denote by $G_n^{(m)}$, so
$$ G_n^{(m)} = \sum_{k_1 + \cdots +k_m=n} \binom{n}{k_1, \ldots, k_m}_q.$$
These generalize the Galois numbers $G_n = G_n^{(2)}$, and in particular, $G_n^{(m)}$ is the total number of flags of subspaces of length $m$ of an $n$-dimensional vector space over a field with $q$ elements when $q$ is the power of a prime.  The main task of Section \ref{flags} is to study the numbers $G_n^{(m)}$ in the context of finite vector spaces, independent of the Rogers-Szeg\"o polynomials, in the spirit of the study of Goldman and Rota.  A recursion for the numbers $G_n^{(m)}$ is immediately obtained in Corollary \ref{FlagRecCor} from the recursion of the multivariate Rogers-Szeg\"o polynomials, for which we provide a combinatorial proof in terms of finite vector spaces.  We prove the recursion for the generalized Galois numbers by obtaining a recursive formula for $q$-multinomial coefficients, which itself implies the recursion for the multivariate Rogers-Szeg\"o polynomials.

\section{Rogers-Szeg\"o polynomials} \label{recursion}

For any $a, r \neq 1$, and $n \geq 1$, we define $(a ; r)_n$ by
$$ (a ; r)_n = (1 - a)(1 - ar) \cdots (1 - ar^{n-1}), $$
and $(a ; r)_0 = 1$.  Define $(a ; r)_{\infty}$ by the formal infinite product
$$ (a ; r)_{\infty} = \prod_{i = 0}^{\infty} (1 - ar^i).$$
For our fixed parameter $q \neq 1$, and any $a \neq 1$, define $(a)_n = (a ; q)_n$ and $(a)_{\infty} = (a ; q)_{\infty}$, so $(q)_n = (1 - q)(1-q^2) \cdots (1 - q^n)$ for $n \geq 1$ as in the introduction.

Then we have \cite[Cor. 2.2]{An76} the following formal identity, originally due to Euler:
\begin{equation} \label{Euler}
\sum_{n = 0}^{\infty} \frac{x^n}{(r ; r)_n} = (x ; r)_{\infty}^{-1}.
\end{equation}
We may use (\ref{Euler}) to give a generating function for the multivariate Rogers-Szeg\"o polynomials.  The following is stated in \cite[Ex. 3.17]{An76}, but we include the proof here for the sake of self-containment.

\begin{lemma} \label{genfn} The multivariate Rogers-Szeg\"o polynomials have the following generating function:
$$ \sum_{n= 0}^{\infty} \frac{H_n(t_1, \ldots, t_{m-1})}{(q)_n} x^n =  (t_1x)_{\infty}^{-1} \cdots (t_{m-1}x)_{\infty}^{-1}(x)_{\infty}^{-1}.$$
\end{lemma}
\begin{proof} By (\ref{Euler}), we have
\begin{align*}
(t_1x)_{\infty}^{-1} \cdots (t_{m-1}x)_{\infty}^{-1}(x)_{\infty}^{-1} & = \left( \sum_{k_1 = 0}^{\infty} \frac{t_1^{k_1}x^{k_1}}{(q)_{k_1}} \right) \cdots \left( \sum_{k_{m-1} = 0}^{\infty} \frac{t_{m-1}^{k_{m-1}}x^{k_{m-1}}}{(q)_{k_{m-1}}} \right) \left( \sum_{k_m = 0}^{\infty} \frac{x^{k_m}}{(q)_{k_m}} \right) \\
 & = \sum_{n=0}^{\infty} \left( \sum_{k_1 + \cdots + k_m = n} \frac{t_1^{k_1} \cdots t_{m-1}^{k_{m-1}}}{(q)_{k_1} \cdots (q)_{k_m}} \right) x^n \\
 & = \sum_{n=0}^{\infty} \left( \sum_{k_1 + \cdots + k_m = n} \binom{n}{k_1, \ldots, k_m}_q t_1^{k_1} \cdots t_{m-1}^{k_{m-1}} \right) \frac{x^n}{(q)_n} \\
 & = \sum_{n= 0}^{\infty} \frac{H_n(t_1, \ldots, t_{m-1})}{(q)_n} x^n,
\end{align*}
as claimed.
\end{proof}

For any finite set of variables $X$, we let $\be_i(X)$ denote the $i$th elementary symmetric polynomial in the variables $X$.  We can now give the recursion of the multivariate Rogers-Szeg\"o polynomials which generalizes (\ref{RogSze}).

\begin{theorem} \label{main}
The Rogers-Szeg\"o polynomials in $m-1$ variables satisfy the following recursion:
$$H_{n+1}(t_1, \ldots, t_{m-1}) = \sum_{i = 0}^{m-1} \be_{i+1}(t_1, \ldots, t_{m-1}, 1) (-1)^{i} \frac{(q)_n}{(q)_{n-i}} H_{n-i}(t_1, \ldots, t_{m-1}).$$
\end{theorem}
\begin{proof} Let 
$$ F(x, t_1, \ldots, t_{m-1}) = \sum_{n=0}^{\infty} \frac{H_n(t_1, \ldots, t_{m-1})}{(q)_n} x^n = (t_1x)_{\infty}^{-1} \cdots (t_{m-1}x)_{\infty}^{-1}(x)_{\infty}^{-1},$$  
by Lemma \ref{genfn}.  Since $(1 - t_ix)(t_ix)_{\infty}^{-1} = (t_ixq)_{\infty}^{-1}$ and $(1-x)(x)_{\infty}^{-1} = (xq)_{\infty}^{-1}$, we have
\begin{align} \label{functional1}
(1 - t_1 x)\cdots (1 - t_{m-1}x) (1-x) F(x, t_1, \ldots, t_{m-1}) & = (t_1 xq)_{\infty}^{-1} \cdots (t_{m-1} xq)_{\infty}^{-1} (xq)_{\infty}^{-1} \notag \\ 
& = F(xq, t_1, \ldots, t_{m-1}).
\end{align}
By the definition of the elementary symmetric polynomials, we have
$$ (1 - t_1 x) \cdots (1 - t_{m-1} x) (1 - x) = \sum_{i = 0}^m (-1)^i \be_i(t_1, \ldots, t_{m-1}, 1) x^i.$$
From this and (\ref{functional1}), we have
$$ \left( \sum_{i = 0}^m (-1)^i \be_i(t_1, \ldots, t_{m-1}, 1) x^i \right)  \sum_{n=0}^{\infty} \frac{H_n(t_1, \ldots, t_{m-1})}{(q)_n} x^n = \sum_{n=0}^{\infty} \frac{q^n H_n(t_1, \ldots, t_{m-1})}{(q)_n} x^n,$$
which may be re-written as
$$ \sum_{n=0}^{\infty} \frac{(1 - q^n) H_n(t_1, \ldots, t_{m-1})}{(q)_n} x^n = \left( \sum_{i = 0}^{m-1} (-1)^i\be_{i+1} (t_1, \ldots, t_{m-1}, 1) x^{i+1} \right) \sum_{n=0}^{\infty} \frac{H_n(t_1, \ldots, t_{m-1})}{(q)_n} x^n.$$
Comparing the coefficients of $x^{n+1}$ in both sides of the above expression, we obtain
$$\frac{H_{n+1}(t_1, \ldots, t_{m-1})}{(q)_n} = \sum_{i = 0}^{m-1} \be_{i+1}(t_1, \ldots, t_{m-1}, 1) \frac{(-1)^i}{(q)_{n-i}} H_{n-i}(t_1, \ldots, t_{m-1}),$$
which yields the desired result.
\end{proof}

Note that we have $(-1)^i (q)_n/(q)_{n-i} = (q^n - 1)(q^{n-1} - 1) \cdots (q^{n-i + 1} - 1)$.  For example, applying Theorem \ref{main} to the Rogers-Szeg\"o polynomials in two variables, we obtain
\begin{align*}
H_{n+1}(t_1, t_2) = (1 + t_1 + t_2) H_n(t_1, t_2)  + & (t_1 t_2 + t_1 + t_2)(q^n - 1) H_{n-1}(t_1, t_2) \\ & + t_1 t_2 (q^n - 1)(q^{n-1} - 1) H_{n-2}(t_1, t_2).
\end{align*}

When the Rogers-Szeg\"o polynomial $H_n(t)$ in a single variable is evaluated at $t=-1$, we get the following identity for the alternating sum of $q$-binomial coefficients originally due to Gauss:
\begin{equation} \label{Special1single}
H_n(-1) = \sum_{k=0}^n \binom{n}{k}_q (-1)^k = \left\{ \begin{array}{ll} \frac{(q)_n}{(q^2 ; q^2)_{n/2}} = \prod_{j<n, j \text{ odd }} (1 - q^j) & \text{ if $n$ is even,} \\ 0 & \text{ otherwise.} \end{array} \right. 
\end{equation} 
As pointed out in \cite[Sec. 21]{Fi88}, this identity may be generalized by evaluating the multivariate Rogers-Szeg\"o polynomials at roots of unity.  That is, if $\omega = e^{2\pi i/m}$ is a primitive $m$-th
  root of unity, where $m \geq 2$, and let $n \geq 0$, then 
\begin{equation} \label{Special1}
H_n(\omega, \omega^2, \ldots, \omega^{m-1}) = \left\{ \begin{array}{ll} \frac{(q)_n}{(q^m ; q^m)_{n/m}} = \prod_{j < n, m \nmid j} (1 - q^j) & \text{ if $m|n$,} \\ 0 & \text{ otherwise.} \end{array} \right. 
\end{equation}

This is calculated in \cite{Fi88} by applying the generating function in Lemma \ref{genfn}.  We note that we may also compute it quickly from Theorem \ref{main} as follows.  We have $\be_i(\omega, \ldots, \omega^{m-1}, 1) = 0$ for $1 \leq i \leq
  m-1$, while $\be_m(\omega, \ldots, \omega^{m-1}, 1) = (-1)^{m+1}$, since these roots of unity are the roots of $x^m - 1$.   We have
  $H_0(t_1, \ldots, t_{m-1}) = 1$, and for $0 < n \leq m-1$, we have $H_n(t_1, \ldots, t_{m-1})$ is a symmetric polynomial in $t_1, \ldots, t_{m-1}, 1$, of degree $n < m$, with zero constant term.  We may thus write $H_n(t_1, \ldots, t_{m-1})$ as a polynomial in $\be_i(t_1, \ldots, t_{m-1}, 1)$, with $0< i < n$.  Now, $H_n(\omega, \ldots, \omega^{m-1})$ can be written as a polynomial in $\be_i(\omega, \ldots, \omega^{m-1}, 1)$, for $1 \leq i \leq m-1$, which are all $0$.  It follows that $H_n(\omega, \ldots, \omega^{m-1}) = 0$ for these $n$.  By Theorem \ref{main} and the values of the elementary symmetric polynomials, if $n \geq m$ then
\begin{align*}
H_{n}(\omega, \ldots, \omega^{m-1}) & = \sum_{i = 0}^{m-1} \be_{i+1}(\omega, \ldots, \omega^{m-1}, 1) (-1)^i \frac{(q)_{n-1}}{(q)_{n-1-i}} H_{n-1-i}(\omega, \ldots, \omega^{m-1}) \\
 & = (1- q^{n-1}) (1 - q^{n-2}) \cdots (1-q^{n-m+1}) H_{n-m}(\omega, \ldots, \omega^{m-1}). \notag
\end{align*}
The values (\ref{Special1}) now follow by induction.

Another value of the Rogers-Szeg\"o polynomial of a single variable is
$$H_n(q^{1/2}) = \frac{(q)_n}{(q^{1/2} ; q^{1/2})_n} = (q^{1/2} ; -q^{1/2})_n = \prod_{j = 1}^n (1 + q^{j/2}).$$
This is generalized in \cite{Fi88} with the value
\begin{equation} \label{Special2}
H_n(q^{1/m}, q^{2/m}, \ldots, q^{(m-1)/m}) = \frac{(q)_n}{(q^{1/m} ; q^{1/m})_n} = \prod_{j = 1}^n (1 + q^{j/m} + \cdots + q^{j(m-1)/m}).
\end{equation}
Fine also gives a generalization of both (\ref{Special1}) and (\ref{Special2}), and applies it to obtain a bi-basic identity \cite[21.4]{Fi88}.

The Rogers-Szeg\"o polynomial $H_n(t)$ also takes the value
\begin{equation} \label{Special3single}
H_n(-q) = \frac{(q)_n}{(q^2 ; q^2)}_{ \lfloor n/2 \rfloor} = \prod_{j \leq n, j \text{ odd}} (1 - q^j),
\end{equation}
which is applied in finding identities involving Hall-Littlewood functions in \cite{Wa06}, for example.  A generalization of (\ref{Special3single}) for the multivariate Rogers-Szeg\"o polynomials is not covered above, and so we obtain one now.  Let $\omega = e^{2\pi i/m}$ be a primitive $m$-th root of unity, where $m \geq 2$, and let $n \geq 0$.  Then
\begin{equation} \label{Special3}
H_n(\omega q, \omega^2 q, \ldots, \omega^{m-1} q) = \frac{(q)_n}{(q^m ; q^m)_{\lfloor n/m \rfloor}} = \prod_{j \leq n, m \nmid j} (1 - q^j),
\end{equation} 
which we calculate as follows.  By Lemma \ref{genfn} and (\ref{Euler}), we have
\begin{align*}
\sum_{n=0}^{\infty} \frac{H_n(\omega q, \ldots, \omega^{m-1} q)}{(q)_n} x^n & = \frac{1}{(x)_{\infty} (x\omega q)_{\infty} \cdots (x\omega^{m-1} q)_{\infty}} \\
& = \frac{1}{1 - x} \prod_{j= 1}^{\infty} \prod_{l = 0}^{m-1} \frac{1}{1 - x \omega^l q^j}  = \frac{1}{1-x} \prod_{j=1}^{\infty} \frac{1} {1 - x^m q^{mj}} \\
& = \frac{1 + x + \cdots + x^{m-1}}{(1 - x^m)(1-x^m q^m)(1 - x^m q^{2m}) \cdots} = \frac{1 + x + \cdots + x^{m-1}}{(x^m ; q^m)_{\infty}} \\
& = (1 + x + \cdots + x^{m-1}) \sum_{k = 0}^{\infty} \frac{1}{(q^m ; q^m)_k} x^{mk}.
\end{align*}
Comparing the coefficients of $x^n$, we obtain (\ref{Special3}).

The value (\ref{Special3single}) of $H_n(-q)$ could also be computed by using (\ref{Special1single}) along with the functional equation \cite[20.64b]{Fi88}
\begin{equation} \label{qfunctionalsingle}
H_{n}(tq) = H_n(t) - t(1-q^n) H_{n-1}(t).
\end{equation} 
The next result, which generalizes (\ref{qfunctionalsingle}) to multivariate Rogers-Szeg\"o polynomials, has a very similar form and proof to Theorem \ref{main}.

\begin{theorem} \label{qfunctional}
Let $m \geq 2$, and let $J \subseteq \{ 1, \ldots, m-1 \}$, where $|J| \neq 0$.  For $1 \leq i \leq m-1$, define $s_i = t_i q$ if $i \in J$, and $s_i = t_i$ otherwise.  Let $\be_i(t_J)$ be the $i$th elementary symmetric polynomial in the set of variables $t_J = \{ t_j \, \mid \, j \in J\}$.  Then for $n \geq |J|$, 
$$H_n(s_1, \ldots, s_{m-1}) = \sum_{i = 0}^{|J|} \be_{i}(t_J) (-1)^i \frac{(q)_n}{(q)_{n-i}} H_{n-i} (t_1, \ldots, t_{m-1}).$$
\end{theorem}
\begin{proof} Let $F(x, t_1, \ldots, t_{m-1})$ be the generating function for $H_n(t_1, \ldots, t_{m-1})$ as in the proof of Theorem \ref{main}.  Then we have
\begin{align} \label{functional2}
\left(\prod_{j \in J} (1 - t_jx) \right) F(x, t_1, \ldots, t_{m-1}) & = \left(\prod_{j \in J} (t_jq x)_{\infty}^{-1}\right) \left(\prod_{i \not\in J, 1 \leq i \leq m-1} (t_i x)_{\infty}^{-1} \right) (x)_{\infty}^{-1} \notag \\ 
& = F(x, s_1, \ldots, s_{m-1}).
\end{align}
Since $\prod_{j \in J} (1 - t_j x) = \sum_{i = 0}^{|J|} (-1)^i \be_i(t_J) x^i$, it follows from (\ref{functional2}) that we have
$$ \sum_{n=0}^{\infty} \frac{H_n(s_1, \ldots, s_{m-1})}{(q)_n} x^n = \left( \sum_{i = 0}^{|J|} (-1)^i \be_i(t_J) x^i \right)  \sum_{n=0}^{\infty} \frac{H_n(t_1, \ldots, t_{m-1})}{(q)_n} x^n.$$
Comparing the coefficient of $x^n$ in both sides of the above gives the result.
\end{proof}

Now we may compute the value (\ref{Special3}) by applying Theorem \ref{qfunctional} in the following way.  Note that since $\omega, \omega^2, \ldots, \omega^{m-1}$, are the roots of $x^m + x^{m-1} + \cdots + 1$, then $\be_i(\omega, \ldots, \omega^{m-1}) = (-1)^i$ for $0 \leq i \leq m-1$.  If we are able to compute the values $H_i(\omega q, \ldots, \omega^{m-1} q)$ for $i < m$, then we use Theorem \ref{qfunctional} with $J = \{1, \ldots, m-1\}$ and $t_i = \omega^i$ to obtain, when $n \geq m$,
$$H_n(\omega q, \ldots, \omega^{m-1} q) = \sum_{i = 1}^{m-1} \frac{(q)_n}{(q)_{n-i}} H_{n-i}(\omega, \ldots, \omega^{m-1}).$$
The values (\ref{Special3}) then follow for $n \geq m$ when plugging in the values (\ref{Special1}).  We can compute $H_n(\omega q, \ldots, \omega^{m-1} q)$ for $n < m$ using Theorem \ref{qfunctional} as well.  For $n = 1$, we begin by taking $J = \{1\}$ and $t_i = \omega^i$ to obtain $H_1(\omega q, \omega^2, \ldots, \omega^{m-1}) = \omega (q - 1)$.  Then take $J = \{2\}$, $t_1 = \omega q$, and $t_i = \omega^i$ for $i > 1$.  Applying Theorem \ref{qfunctional} then gives $H_1(\omega q, \omega^2 q, \omega^3, \ldots, \omega^{m-1}) = (\omega + \omega^2)(q-1)$.  Continuing in this way, we get
$$H_1(\omega q, \omega^2 q, \ldots, \omega^{m-1} q) = (\omega + \cdots + \omega^{m-1}) (q-1) = (1-q).$$
The values for $1 < n < m$ may be computed similarly.

\section{Flags in finite vector spaces} \label{flags}

Now let $q$ be the power of a prime, and let $\FF_q$ denote a finite field with $q$ elements.  If $V$ is an $n$-dimensional vector space over $\FF_q$, then the $q$-binomial coefficient $\binom{n}{k}_q$ is the number of $k$-dimensional subspaces of $V$ (see \cite[Thm. 7.1]{KaCh02} or \cite[Prop. 1.3.18]{St97}).  When we evaluate the Rogers-Szeg\"o polynomial at $t=1$, we obtain
$$ H_n(1) = \sum_{k=0}^n \binom{n}{k}_q,$$
which is the total number of subspaces of an $n$-dimensional vector space over $\FF_q$.  We define the Galois numbers as $G_n = H_n(1)$.  As mentioned in the introduction, the recursion for the Rogers-Szeg\"o polynomials (\ref{RogSze}) gives the following recursion for the Galois numbers, which was studied by Goldman and Rota \cite{GoRo69}:
\begin{equation} \label{GalRec}
G_{n+1} = 2G_n + (q^n - 1)G_{n-1}, \quad G_0 = 1, G_1 = 2.
\end{equation}
The recursion (\ref{GalRec}) was proved bijectively by counting subspaces of finite vector spaces by Nijenhuis, Solow, and Wolf \cite{NiSoWi84}.  The proof in \cite{NiSoWi84} is obtained by proving the following result bijectively, from which (\ref{GalRec}) follows.

\begin{lemma} \label{GalLemma}
For integers $n \geq k \geq 1$, we have
$$\binom{n+1}{k}_q = \binom{n}{k}_q + \binom{n}{k-1}_q + (q^n-1)\binom{n-1}{k-1}_q.$$
\end{lemma}

We now consider the meaning of a $q$-multinomial coefficient in terms of vector spaces over $\FF_q$.  It follows from the definition of a $q$-multinomial coefficient and the fact that $\binom{n}{k}_q = \binom{n}{n-k}_q$ that we have
\begin{align*}
\binom{n}{k_1, k_2, \ldots, k_m}_q & = \binom{n}{k_1}_q \binom{n-k_1}{k_2}_q \cdots \binom{n-k_1-\cdots - k_{m-2}}{k_{m-1}}_q\\
& = \binom{n}{n-k_1}_q \binom{n-k_1}{n-k_1-k_2}_q \cdots \binom{n-k_1 - \cdots -k_{m-2}}{n-k_1 -\cdots -k_{m-2} - k_{m-1}}_q.
\end{align*}
So, if $V$ is an $n$-dimensional vector space over $\FF_q$, the $q$-multinomial coefficient $\binom{n}{k_1, \ldots, k_m}_q$ is equal to the number of ways to choose an $(n-k_1)$-dimensional subspace $W_1$ of $V$, an $(n-k_1 -k_2)$-dimensional subspace $W_2$ of $W_1$, and so on, until finally we choose an $(n-k_1-\cdots -k_{m-1})$-dimensional subspace $W_{m-1}$ of some $(n-k_1 - \cdots -k_{m-2})$-dimensional subspace $W_{m-2}$ (see also \cite[Sec. 1.5]{Mo06}).  That is, 
$$W_{m-1} \subseteq W_{m-2} \subseteq \cdots \subseteq W_2 \subseteq W_1$$
is a {\em flag} of subspaces of $V$ of length $m-1$, where ${\rm dim} \; W_i = n - \sum_{j=1}^i k_j$.

If we evaluate the Rogers-Szeg\"o polynomial in $m-1$ variables at $t_1 = t_2 = \cdots = t_{m-1} = 1$, we obtain
$$ H_n(1, 1, \ldots, 1) = \sum_{k_1 + \cdots + k_m = n} \binom{n}{k_1, \ldots, k_m}_q,$$
which, by the discussion above, counts the total number of flags of subspaces of length $m-1$ in an $n$-dimensional $\FF_q$-vector space.  We denote this quantity by $G_n^{(m)}$, so that the Galois number $G_n = G_n^{(2)}$.  We may apply Theorem \ref{main} to obtain a recursion for the numbers $G_n^{(m)}$, generalizing the recursion in (\ref{GalRec}), by noticing that the number of terms in the elementary symmetric polynomial $\be_{i+1}(t_1, \ldots, t_{m-1}, 1)$ is $\binom{m}{i+1}$.

\begin{corollary} \label{FlagRecCor}
The numbers $G_n^{(m)}$ satisfy the following recursion, for $n \geq m-1$:
$$G_{n+1}^{(m)} = \sum_{i=0}^{m-1} \binom{m}{i+1} (-1)^i \frac{(q)_n}{(q)_{n-i}}G_{n-i}^{(m)}.$$
\end{corollary}

In this section, we prove Corollary \ref{FlagRecCor} combinatorially in terms of finite vector spaces, by proving an analog of Lemma \ref{GalLemma}.

We need some notation.  Let $\underline{k}$ denote the $m$-tuple $(k_1, \ldots, k_m)$, and write the corresponding $q$-multinomial coefficient as
$$ \binom{n}{k_1, \ldots, k_m}_q = \binom{n}{\ulk}_q.$$
For a subset $J \subseteq \{ 1, \ldots, m \}$, let $\ule_J$ denote the $m$-tuple $(e_1, \ldots, e_m)$, where 
$$ e_i = \left\{ \begin{array}{ll} 1 & \text{ if $i \in J$,} \\ 0 & \text{ if $i \not\in J$.} \end{array} \right.$$
For example, if $m = 3$, $J = \{1, 3\}$, and $\ulk = (k_1, k_2, k_3),$ then
$$\binom{n}{\ulk - \ule_J}_q = \binom{n}{k_1 -1, k_2, k_3 - 1}_q.$$
The following is our generalization of Lemma \ref{GalLemma}.

\begin{lemma} \label{GenGalLemma}
For $m \geq 2$, and any $k_1, \ldots, k_m > 0$ such that $k_1 + \cdots + k_m = n+1$, we have
$$ \binom{n+1}{k_1, \ldots, k_m}_q = \sum_{J \subseteq \{1, \ldots, m\}, |J| > 0} (-1)^{|J| - 1} \frac{(q)_{n}}{(q)_{n - |J| + 1}} \binom{n + 1 - |J|}{\ulk - \ule_J}_q$$
\end{lemma}

Before proving Lemma \ref{GenGalLemma}, we explain why it implies Corollary \ref{FlagRecCor}.  First note that we may get a version of Lemma \ref{GenGalLemma} which allows any of the $k_i = 0$ as follows.  If we want $l$ of the $k_i$'s to be $0$, we start with applying Lemma \ref{GenGalLemma} to a $q$-multinomial coefficient of length $m-l$, using the $m-l$ nonzero $k_i$'s, and note that the equation in Lemma \ref{GenGalLemma} is not affected by inserting $0$'s into the appropriate positions of all the $q$-multinomial coefficients in both sides.  That is, if some $k_i=0$, we may still apply Lemma \ref{GenGalLemma}, while ignoring these $k_i$, or equivalently, we may apply Lemma \ref{GenGalLemma} to the $q$-multinomial coefficient obtained by removing the $k_i$'s which are $0$, and re-inserting these $0$'s in all $q$-multinomial coefficients in the sum the end.

Now consider the sum of all $q$-multinomial coefficients of the form $\binom{n+1}{k_1, \ldots, k_m}_q$, while applying the more general version of Lemma \ref{GenGalLemma} just discussed.  For any $i$, $0 \leq i \leq m-1$, each term $\binom{n+1}{k_1, \ldots, k_m}_q$ in $G_{n+1}^{(m)}$ may be obtained by adding $1$ to $i+1$ of the $l_i$'s in terms in $G_{n-i}^{(m)}$ of the form $\binom{n-i}{l_1, \ldots, l_m}_q$ in exactly $\binom{m}{i+1}$ ways.  By Lemma \ref{GenGalLemma}, these terms contribute exactly what we need to conclude Corollary \ref{FlagRecCor}.  By a similar argument, we may see that in fact the recursion for the multinomial Rogers-Szeg\"o polynomials in Theorem \ref{main} also follows from Lemma \ref{GenGalLemma}.

\begin{proof}[Proof of Lemma \ref{GenGalLemma}]  We will prove this by induction on $m$, where the base case $m=2$ is given by Lemma \ref{GalLemma}.  Fix $V$ to be an $(n+1)$-dimensional vector space over $\FF_q$.  We know $\binom{n+1}{k_1, \ldots, k_m}_q$ is the number of flags of subspaces of $V$, $W_{m-1} \subset \cdots \subset W_2 \subset W_1$, where ${\rm dim} \; W_i = n + 1 - \sum_{j=1}^i k_j$.  We must show that the right-hand side of the claimed identity in Lemma \ref{GenGalLemma} also counts these flags.  We have, by choosing first the subspace $W_1$ and then the rest of the flag, and applying Lemma \ref{GalLemma},
\begin{align} \label{firststep}
\binom{n+1}{k_1, \ldots, k_m}_q & = \binom{n+1}{n+1 - k_1}_q \binom{n+1 - k_1}{k_2, \ldots, k_m}_q \notag\\
& = \left( \binom{n}{n+1 - k_1}_q + \binom{n}{n-k_1}_q + (q^n - 1) \binom{n-1}{n - k_1}_q \right) \binom{n+1 - k_1}{k_2, \ldots, k_m}_q.
\end{align}
We have $\binom{n}{n+1-k_1}_q \binom{n+1 - k_1}{k_2, \ldots, k_m}_q = \binom{n}{k_1 - 1, k_2, \ldots, k_m}_q$, which is the term corresponding to the subset $J = \{1\} \subset \{1, \ldots, m\}$.  This may be thought of as the total number of ways of choosing our flag so that $W_1$ is contained in some fixed $n$-dimensional subspace of $V$.  By our induction hypothesis, the number of remaining flags is given by
\begin{equation} \label{remainingflags}
\left(\binom{n}{n-k_1}_q + (q^n - 1) \binom{n-1}{n - k_1}_q \right)\sum_{I \subseteq \{1, \ldots, m-1\}, |I| > 0} (-1)^{|I| - 1} \frac{(q)_{n-k_1}}{(q)_{n -k_1 - |I| + 1}} \binom{n + 1 - k_1 - |I|}{\ulk' - \ule_{I}}_q,
\end{equation}
where $\ulk' = (k_2, \ldots, k_m)$.  We have
$$\binom{n}{n-k_1}_q \frac{(q)_{n-k_1}}{(q)_{n - k_1 - |I| + 1}} = \frac{(q)_n}{(q)_{n - |I| + 1}} \binom{n+1 - |I|}{n - k_1 - |I| + 1}_q,$$
and
$$(q^n - 1) \binom{n-1}{n - k_1}_q\frac{(q)_{n-k_1}}{(q)_{n -k_1 - |I| + 1}} = (-1) \frac{(q)_n}{(q)_{n - |I|}} \binom{n-|I|}{n - k_1 - |I| + 1}_q.$$
Given $I \subset \{1, \ldots, m-1\}$, $|I| > 0$, write $I + 1 = \{i+1  \mid  i \in I\} \subseteq \{2, \ldots, m\}$.  If we let $\ulk = (k_1, \ldots, k_m)$, we now have
\begin{equation} \label{Term1}
\binom{n}{n-k_1}_q(-1)^{|I| - 1} \frac{(q)_{n-k_1}}{(q)_{n -k_1 - |I| + 1}} \binom{n + 1 - k_1 - |I|}{\ulk' - \ule_{I}}_q = (-1)^{|J| - 1} \frac{(q)_{n}}{(q)_{n - |J| + 1}} \binom{n + 1 - |J|}{\ulk - \ule_J}_q,
\end{equation}
where $J = I + 1$, $|J| = |I|$, and 
\begin{align} \label{Term2}
(q^n - 1) \binom{n-1}{n - k_1}_q (-1)^{|I| - 1} \frac{(q)_{n-k_1}}{(q)_{n -k_1 - |I| + 1}} & \binom{n + 1 - k_1 - |I|}{\ulk' - \ule_{I}}_q \notag \\ 
& = (-1)^{|J| - 1} \frac{(q)_{n}}{(q)_{n - |J| + 1}} \binom{n + 1 - |J|}{\ulk - \ule_J}_q,
\end{align}
where $J = \{1\} \cup (I + 1)$, $|J| = |I| + 1$.  As $I$ ranges over nonempty subsets of $\{1, \ldots, m-1\}$, $I + 1$ and $\{ 1 \} \cup (I+1)$ range over all nonempty subsets of $\{1, \ldots, m\}$ other than $\{ 1 \}$.  Finally, we substitute (\ref{Term1}) and (\ref{Term2}) into (\ref{remainingflags}), and we see that the sum of all of these terms, along with $\binom{n+1}{k_1 - 1, k_2, \ldots, k_m}$ corresponding to $J = \{ 1 \}$, gives the desired result.
\end{proof}

While the inductive proof of Lemma \ref{GenGalLemma} above works nicely, it somewhat disguises the way we count our flags to see the result bijectively.  We conclude with an explanation of this count.  First, we need to understand the combinatorial proof of Lemma \ref{GalLemma} appearing in \cite{NiSoWi84}, which may be summarized as follows.  Fix $V$ to be an $(n+1)$-dimensional $\FF_q$-vector space as before.  There are $\binom{n+1}{k}_q$ ways to choose a $k$-dimensional subspace $W$ of $V$.  Fix a basis $\{v_1, v_2, \ldots, v_{n+1}\}$ of $V$.  Any $k$-dimensional subspace $W$ can be written as $\mathrm{span}(W', v)$ where $W'$ is a $(k-1)$-dimensional subspace of $V' = \mathrm{span}(v_1, \ldots, v_n)$.  We may choose $W$ in three distinct ways.  If $v \in V'$, then $W$ is a subspace of $V'$, for which there are $\binom{n}{k}_q$ choices.  Call this a \emph{Type 1} subspace of $V$.  If we take $v$ to be a scalar multiple of $v_{n+1}$, then $W$ is determined by $W'$, for which there are $\binom{n}{k-1}_q$ choices.  We call this a \emph{Type 2} subspace of $V$.  Finally, if $v$ is neither in $V'$ nor a scalar multiple of $v_{n+1}$, then we call $W$ a \emph{Type 3} subspace of $V$, and it can be shown that there are $(q^n-1)\binom{n-1}{k-1}_q$ choices for $W$, giving Lemma \ref{GalLemma}.

We now fix a basis of \emph{every} subspace $U$ of $V$, so that we may speak of subspaces of Type 1, 2, or 3 of $U$.  Consider a flag of subspaces of $V = W_0$, $W_{m-1} \subset \cdots \subset W_2 \subset W_1$, such that if we define $k_i$ for $1 \leq i \leq m$ by $\sum_{j=1}^i k_j = n + 1 - {\rm dim} \; W_i$, then each $k_i > 0$.  The total number of such flags is $\binom{n+1}{k_1, \ldots, k_m}_q$, and these flags may also be counted in the following way.  We may choose $W_1$ to be a Type 1 subspace of $V$, or we may choose every $W_i$ to be a Type 2 or Type 3 subspace of $W_{i - 1}$ for $i \leq m-1$, or we may choose $W_i$ to be a Type 2 or Type 3 subspace of $W_{i-1}$ for $i \leq r-1$ for some $r < m$ and $W_r$ a Type 1 subspace of $W_{r-1}$.  These cases account for all possibilities for such a flag of $V$.  For a nonempty $J \subseteq \{ 1, \ldots, m \}$, let $r$ be the maximum element of $J$.  Then a closer look at the proof of Lemma \ref{GenGalLemma} reveals that
$$(-1)^{|J| - 1} \frac{(q)_{n}}{(q)_{n - |J| + 1}} \binom{n + 1 - |J|}{\ulk - \ule_J}_q$$
is the number of ways to choose our flag such that $W_j$ is a Type 3 subspace of $W_{j-1}$ for $j \in J$ and $j < r$, $W_i$ is a Type 2 subspace of $W_{i-1}$ for $i \not\in J$ and $i < r$, and $W_r$ is a Type 1 subspace of $W_{r-1}$ if $r < m$.  These account for all $2^m - 1$ terms in the right-side of the equation in Lemma \ref{GenGalLemma}, and all possible ways to choose our flag.
\\
\\
\noindent
{\bf Acknowledgments.  }  The author thanks Prof. George Andrews for helpful and encouraging comments.  This research was supported by NSF grant DMS-0854849.

\bigskip

\noindent
\begin{tabular}{ll}
\textsc{Department of Mathematics}\\ 
\textsc{College of William and Mary}\\
\textsc{P. O. Box 8795}\\
\textsc{Williamsburg, VA  23187}\\
{\em e-mail}:  {\tt vinroot@math.wm.edu}\\
\end{tabular}

\end{document}